\documentclass[reqno]{amsart}
\usepackage{amsmath,amsfonts}
\usepackage[inner=1.0in,outer=1.0in,bottom=1.0in, top=1.0in]{geometry}
\newtheorem{theorem}{Theorem}[section]

\newtheorem{corollary}[theorem]{Corollary}
\newtheorem{lemma}[theorem]{Lemma}
\newtheorem{remark}[theorem]{Remark}

\numberwithin{equation}{section}
\begin{document}
\title[Bounding the number of nodal domains on the square]{Bounding the number of nodal domains of eigenfunctions without singular points on the square}
\author{Junehyuk Jung}
\address{Department of Mathematics, Texas A\&M University, College Station, TX 77840}
\email{junehyuk@math.tamu.edu}
\begin{abstract}
We prove Polterovich's conjecture concerning the growth of the number of nodal domains for eigenfunctions on a unit square domain, under the assumption that the eigenfunctions do not have any singular points.
\end{abstract}
\thanks{We would like to thank Jean Bourgain for introducing the problem, and encouragement. We appreciate Stefan Steinerberger, Van Vu, Igor Wigman, Steve Zelditch, and Joon-Hyeok Yim for many helpful discussions. We also thank Bernard Helffer and the anonymous referee for detailed comments on the earlier version of the manuscript.}
\maketitle
\section{Introduction}
Let $\Omega \subset \mathbb{R}^2$ be a domain with a piecewise smooth boundary. Let $\{\phi_j\}_{j\geq 1}$ be an orthonormal Dirichlet eigenbasis with the eigenvalues $0<\lambda_1 < \lambda_2 \leq \lambda_3 \ldots$, i.e.,
\begin{align*}
-\Delta \phi_j &= \lambda_j \phi_j\\
\phi_j (x) &= 0 \text{ for } x\in  \partial \Omega\\
\langle \phi_j , \phi_k \rangle  &= \delta_{jk}
\end{align*}
where $\Delta = \partial_x^2+\partial_y^2$ is the Laplacian on $\mathbb{R}^2$, and $\delta_{jk}$ is the Kronecker delta. We assume that all $\phi_j$'s are real-valued.

Denoting by $N(\phi_j)$ the number of nodal domains (connected components of $\Omega$ minus the zero set of $\phi_j$), Courant's nodal domain theorem \cite{ch53} implies that
\begin{equation}\label{courant}
N(\phi_j) \leq j
\end{equation}
for all $j\geq 1$. This inequality is sharp, and we say that $\lambda_j$ a Courant-sharp Dirichlet eigenvalue, whenever the equality is satisfied by $\phi_j$. For instance, $\lambda_1$, $\lambda_2$ are Courant-sharp Dirichlet eigenvalues.

For large values of $j$, \eqref{courant} is not sharp anymore. In particular, Pleijel's theorem \cite{pl56} states that
\begin{equation}\label{pleijel}
\limsup_{j\to \infty} \frac{N(\phi_j)}{j} \leq \left(\frac{2}{j_{0,1}}\right)^2 = 0.69167\ldots,
\end{equation}
where $j_{0,1}$ is the smallest positive zero of the Bessel function, $J_0(x)$.

Note that when $\Omega$ is a unit square domain $D:=\lbrack 0,1 \rbrack ^2 \subset \mathbb{R}^2$, one can prove that
\[
\limsup_{j\to \infty} \frac{N(\phi_j)}{j} \geq \frac{2}{\pi} =0.63661\ldots
\]
by considering a sequence of Dirichlet eigenfunctions $\{\sin (k \pi x)\sin (k \pi y)\}_{k=1,2,\ldots}$.

In \cite{po09}, Polterovich observed that Pleijel's inequality should not be sharp, because it uses Faber--Krahn inequality which is only sharp on a disk, whereas not all nodal domains can simultaneously be disks. He also conjectured that the maximum number of nodal domains are obtained by the eigenfunctions $\sin (k \pi x)\sin (k \pi y)$ on $\Omega=D$, i.e.,
\begin{equation} \label{polterovich}
\limsup_{j\to \infty} \frac{N(\phi_j)}{j} \leq \frac{2}{\pi}.
\end{equation}

The idea that not all nodal domains can simultaneously be disks
was used by Bourgain \cite{bo15} to improve  Pleijel's inequality, see  also \cite{st14}. In \cite{bo15}, the packing density of disks, and a refined Faber--Krahn inequality are used to prove that
\[
\limsup_{j\to \infty} \frac{N(\phi_j)}{j} \leq \left(\frac{2}{j_{0,1}}\right)^2 - 3\cdot 10^{-9}.
\]
\cite{st14} utilizes a ``geometric uncertainty principle'' that quantifies the fact that if a partition is given by sets of equal measure, then the sets cannot be disks, and proves the existence of a small constant $\eta>0$, such that
\[
\limsup_{j\to \infty} \frac{N(\phi_j)}{j} \leq \left(\frac{2}{j_{0,1}}\right)^2 - \eta.
\]

The main purpose of this article is to present a new way of counting the number of nodal domains. In particular, we give a proof of Polterovich's conjecture \eqref{polterovich} for certain eigenfunctions on a unit square domain $D$. (We refer the readers to \cite{helffer15} for the case of irrational rectangles.)

\begin{theorem}\label{thm}
Let $0<\lambda_1 < \lambda_2 \leq \lambda_3 \leq \ldots$ be the complete spectrum of Laplacian on a unit square domain $D= \lbrack 0,1 \rbrack^2 \subset \mathbb{R}^2$. Let $\phi_j$ be a real valued Dirichlet eigenfunction on $D$ with the eigenvalue $\lambda_j$. Assume that $\phi_j$ does not have any interior singular point, where both the value and the gradient vanishes. Then we have
\[
\frac{N(\phi_j)}{j} \leq \frac{2}{\pi} + O(j^{-1/2}),
\]
as $j \to \infty$.
\end{theorem}
\begin{remark}
For any coprime integers $a,b$, and a real parameter $t\in \mathbb{R}$, $\sin (\pi a x) \sin (\pi b y) + t\sin (\pi b x) \sin (\pi a y)$ is a real valued Dirichlet eigenfunction on $D$ with the eigenvalue $\pi^2(a^2+b^2)$. It has no interior singular point for all but finitely many $t\in \mathbb{R}$.
\end{remark}
The main idea of the proof is to construct a mesh, where most of the connected components of the nodal set have to intersect at least twice. By counting the number of zeros of the eigenfunctions on the mesh, we find an upper bound for the number of connected components of the nodal set. When an eigenfunction does not have any singular point, the number of connected components of the nodal set $+1$ is the number of nodal domains on $D$, from which Theorem \ref{thm} follows. We describe details in the following section.

In Section \ref{3}, we demonstrate how one can improve the counting and go beyond \eqref{polterovich} by considering different types of meshes for eigenfunctions on a flat torus. 

In Section \ref{4}, we consider nodal domains of eigenfunctions on a closed surface. In particular, we prove a formula that relates the number of nodal domains with the number of singular points and the number of connected components of the nodal set. As an application, we give a quick proof of an upper bound for the number of singular points that is first obtained in \cite{df90} and \cite{dong92}.

\section{Proof of the main theorem}
Let $Z_\phi = \{0<x,y<1~:~ \phi(x,y) = 0\}$ be the interior nodal set. Then each connected component of $Z_\phi$ is either a closed curve or a segment touching the boundary $\partial D$. We label connected components of $Z_\phi$ those are segments  by $Z_{\phi,s} (1) , \ldots , Z_{\phi,s} (N_s)$, and those are closed curves by $Z_{\phi,c} (1) , \ldots, Z_{\phi,c} (N_c)$. Firstly, note that we have the following equality for the number of nodal domains of $\phi$
\[
N(\phi) = N_s+N_c+1.
\]

Now we consider a mesh $M$ in $D$ given by
\[
M_\tau = \bigcup_{k=1}^{\left\lfloor \frac{\sqrt{\tau-\pi^2}}{\pi}\right\rfloor} \left\{\left(\frac{\pi k}{\sqrt{\tau-\pi^2}},y\right) ~:~ 0\leq y \leq 1\right\},
\]
which is the nodal set of $\sin \sqrt{\tau-\pi^2} x \sin \pi y$. We may choose $\tau > \lambda$ such that $M_\tau$ intersect $Z_\phi$ transversally. (Note that $\tau$ can be chosen arbitrarily close to $\lambda$.)
\begin{lemma}
Any $Z_{\phi,c}(j)$ should intersect $M_\tau$ at least twice.
\end{lemma}
\begin{proof}
If some $Z_{\phi,c}(j)$ does not intersect $M_\tau$, then we have a connected component of $D - M_\tau$ that contains a nodal domain of $\phi$. This is contradiction to $\tau>\lambda$. Since $Z_{\phi,c}(j)$ is a closed curve and any connected component of $M_\tau$ divides $D$ into two,  the intersection has to occur at least twice.
\end{proof}

\begin{lemma}\label{lem}
$\phi$ has at most $\frac{\sqrt{\lambda}}{\pi}$ sign changes on any fixed component of $M_\tau$.
\end{lemma}
\begin{proof}
$\phi$ restricted to any connected component of $M_\tau$ is a linear combination of $2i \sin (\pi m y)=e^{\pi i m y}-e^{-\pi i m y}$ where $1\leq m \leq \frac{\sqrt{\lambda}}{\pi}$, which is a degree $ < \frac{\sqrt{\lambda}}{\pi}$ polynomial in $e^{\pi i y}+e^{-\pi i y}$ when divided by $e^{\pi i  y}-e^{-\pi i  y}$. Therefore it has at most $\frac{\sqrt{\lambda}}{\pi}$ sign changes.
\end{proof}

From these two lemma, we conclude that
\begin{equation}\label{eq1}
2N_c \leq  \frac{\sqrt{\tau-\pi^2}}{\pi} \frac{\sqrt{\lambda}}{\pi},
\end{equation}
for any $\tau>\lambda$.

In order to bound $N_s$, we first note that each $Z_{\phi,s}$ has two end points on $\partial D$. Also note that there exists a sufficiently small $\epsilon>0$, such that the number of end points on
\[
\{(0,y)~:~0 \leq y \leq 1\}
\]
is equal to the number of zeros of $\phi$ on
\[
\{(\epsilon,y)~:~0 < y < 1\}.
\]
It follows from the same argument in Lemma \ref{lem} applied to this segment that there are at most $\frac{\sqrt{\lambda}}{\pi}$ end points of $Z_{\phi,s}$ on
\[
\{(0,y)~:~0 \leq y \leq 1\},
\]
hence there are at most $\frac{4\sqrt{\lambda}}{\pi}$ end points on $\partial D$. Therefore
\[
N_s \leq \frac{2\sqrt{\lambda}}{\pi},
\]
and combining with \eqref{eq1}, we have
\[
N(\phi) \leq \frac{\lambda}{2\pi^2} + \frac{2\sqrt{\lambda}}{\pi}+1.
\]
Now Theorem \ref{thm} follows from the Weyl law $ j = \frac{\lambda}{4\pi}+O(\sqrt{\lambda})$.

\section{Further development}\label{3}
To begin with, we define $\mathcal{N}(\theta)$ to be the set of positive integer $n$ such that there exists a sector $S_{\alpha, \theta}:=\{(r,\varphi) \in \mathbb{R}^2~:~ \alpha<\varphi <\alpha+\theta, r\geq 0\}$ for some $\alpha$ so that no solution $(a,b) \in \mathbb{Z}^2$ to $a^2+b^2=n$ is contained in $S_{\alpha, \theta}$. For instance, if the number of distinct solutions to $a^2+b^2=n$ is less than $\frac{2\pi}{\theta}$, then $n$ is in $\mathcal{N}(\theta)$.
\begin{theorem}\label{ext}
Fix $\theta$ and $\epsilon$, such that $\theta>\epsilon>0$. Let $0<\lambda_1 < \lambda_2 \leq \lambda_3 \leq \ldots$ be the complete spectrum of Laplacian on a flat torus $\mathbb{T}= \mathbb{R}^2/\mathbb{Z}^2$. Let $\phi_j$ be a real valued Laplacian eigenfunction on $\mathbb{T}$ with the eigenvalue $\lambda_j$. Assume that $\phi_j$ does not have any singular point, and assume further that $n_j:=\lambda_j/(4\pi^2) \in \mathcal{N}(2\theta)$. Then we have that
\[
\frac{N(\phi_j)}{j} \leq \frac{2}{\pi} \cos (\theta-\epsilon) +O_\epsilon(j^{-1/2}).
\]
\end{theorem}
\begin{remark}
For any pair of integers $a$ and $b$, $\cos 2\pi(ax+by)$ is an eigenfunction that has no singular point. (In fact, in each eigenspace $E_\lambda$, the set of eigenfunctions with at least one singular point has codimension $\geq 1$ \cite[Lemma 2.3]{wigman}.)  So the Theorem applies to $\cos 2\pi (ax+by)$ with $\theta = \pi/9$, whenever $a^2+b^2=p$ is a prime $\equiv 1 \pmod{4}$, since $x^2+y^2=p$ has exactly $8$ integral solutions.
\end{remark}
\begin{remark}
The existence of a sequence of eigenfunctions on $\mathbb{T}$ for which we have stronger upper bound than \eqref{polterovich} was known in \cite{bo14}. In \cite{bo14}, Bourgain adopted the idea of Nazarov--Sodin \cite{ns} to a certain sequence of eigenfunctions on $\mathbb{T}$ and proved that
\[
\frac{N(\phi_j)}{j} \sim \kappa
\]
for such eigenfunctions $\phi_j$, where $\kappa>0$ is the constant given in \cite{ns} (it is known that $\kappa \leq \frac{1}{\sqrt{2}\pi} <\frac{2}{\pi}$ \cite{ber,K}). One of the necessary conditions that the eigenfunctions $\phi_j$ should satisfy is that the set
\[
\left\{\left(\frac{a}{\sqrt{a^2+b^2}},\frac{b}{\sqrt{a^2+b^2}}\right)~:~a,b\in \mathbb{Z}, ~n_j =a^2+b^2\right\}
\]
being equidistributed on the unit circle as $j\to \infty$.  Note that this is mutually exclusive to the condition that the eigenfunctions in Theorem \ref{ext} satisfy.
\end{remark}

\begin{proof}
We first fix a finite set of points $P(\epsilon)$ in $\mathbb{Z}^2$ such that $S_{\alpha, \epsilon}$ contains at least one element of $P(\epsilon)$ for any $\alpha$. We may assume without loss of generality that $n_j>x^2+y^2$ for any $(x,y) \in P(\epsilon)$.

From the assumption that $n_j \in \mathcal{N}(2\theta)$, there exists a point $(p,q) \in P(\epsilon)$ such that
\[
|(a,b)(p,q)| \leq \sqrt{n_j}\sqrt{p^2+q^2} \cos (\theta -\epsilon)
\]
for any $(a,b)$ with $a^2+b^2 = n_j$.

Writing
\[
\phi_j (x,y)= \sum_{a^2+b^2 =n_j} \alpha_{ab} e^{2\pi i (a,b)(x,y)},
\]
we see that for any fixed $\tau \geq 0$,
\[
\phi_j (p(t-\tau), qt) = \sum_{a^2+b^2 =n_j} \alpha_{ab}e^{ - 2\pi i a\tau p} e^{2\pi i (a,b)(p,q) t}
\]
is a trigonometric polynomial of length $\leq 2\sqrt{n_j}\sqrt{p^2+q^2} \cos (\theta-\epsilon) +1$.

Now let $C(p,q)$ be a closed geodesic on $\mathbb{T}$, that is an image of
\[
\{(pt, qt) ~:~ 0\leq t\leq 1\} \in \mathbb{R}^2.
\]
For any small $\tau>0$ we let $M(p,q)_\tau$ be the union of shifts of $C(p,q)$,
\[
C(p,q)+ (\tau,0),~ C(p,q)+ (2\tau,0) , \cdots, ~C(p,q) + \left(\left\lceil \frac{1}{q\tau}\right\rceil \tau,0\right),
\]
and the image of
\[
E=\{(x,0)~:~0\leq x \leq 1\} \cup \{(0,y)~:~0\leq y \leq 1\} \subset \mathbb{R}^2
\]
in $\mathbb{T}$.

Considering
\[
\sin \frac{\pi}{\tau}(x - \frac{p}{q}y),
\]
we see that the region in $\mathbb{T}$ bounded by $C(p,q)+k\tau$ and $C(p,q)+(k+1)\tau$ has the first Dirichlet eigenvalue
\[
\frac{\pi^2}{\tau^2} (1+\frac{p^2}{q^2}).
\]
Therefore if we choose $\tau$ such that
\[
\frac{\pi^2}{\tau^2} (1+\frac{p^2}{q^2}) = \lambda_j,
\]
then any connected component of the nodal set of $\phi_j$ should intersect $M(p,q)_\tau$ at least twice.

Denoting by $C(\phi_j)$ the number of connected components of the nodal set of $\phi_j$, we have
\[
2C(\phi_j) \leq 2 \sqrt{n_j}\sqrt{p^2+q^2} \cos (\theta-\epsilon) \left\lceil \frac{1}{q\tau} \right\rceil + 4\sqrt{n_j}
\]
where the latter term amounts to the intersection between the nodal set and $E$.

Note that when $\phi_j$ has no singular points, we have $C(\phi_j)+1 \geq N(\phi_j)$, we therefore conclude that
\begin{align*}
N(\phi_j) &\leq  \sqrt{n_j}\sqrt{p^2+q^2} \cos (\theta-\epsilon) \left\lceil \frac{1}{q\tau} \right\rceil + 2\sqrt{n_j}\\
&< \sqrt{n_j}\sqrt{p^2+q^2} \cos (\theta-\epsilon) \left( \frac{1}{q\tau} +1\right)+2\sqrt{n_j}\\
&<\frac{\sqrt{n_j \lambda_j}}{\pi} \cos (\theta-\epsilon) + O_\epsilon (n_j^{1/2})\\
&=\frac{\lambda_j}{2\pi^2} \cos (\theta-\epsilon) +O_\epsilon (\lambda_j^{1/2}),
\end{align*}
and the theorem follows from the Weyl law $ j = \frac{\lambda_j}{4\pi}+O(\sqrt{\lambda_j})$.
\end{proof}
\section{On closed surfaces}\label{4}
We first prove an inequality for a graph embedded on a closed surface.
\begin{lemma}\label{lemma1}
Let $G$ be a graph embedded on a closed surface $M$ with genus $\mathfrak{g}$. Let $v$, $e$, $f$, and $c$ be the number of vertices, edges, faces, and connected components of $G$. Then we have
\[
1 \geq v-e+f-c \geq 1-2\mathfrak{g}.
\]
\end{lemma}
\begin{proof}
The right hand side is (6.1) of \cite{jz2016}. To prove the left hand side inequality, we first note that when $\mathfrak{g}=0$, we have the equality
\[
v-e+f-c=1.
\]
We are going to prove the other cases by induction on $c+\mathfrak{g}$. Assume that the inequality is true for all $c$ and $\mathfrak{g}$ such that $c+\mathfrak{g} =N$. For a pair of a graph $G$ and a closed surface $M$, we consider two cases:
\begin{enumerate}
  \item Every face is homeomorphic to a disk.
  \item Otherwise.
\end{enumerate}
In the first case, we have $c=1$, and $v-e+f =2-2\mathfrak{g}$ (Euler's theorem), so the inequality is satisfied.

In the second case, we can find a face of $G$ that contains a closed curve $C$ that is not contractible within the face. $M-C$ is a surface with two boundary components, so we consider a closed surface $\tilde{M}$ obtained by attaching two caps on each boundary component. If $\tilde{M}$ is connected, then we have $v=\tilde{v}$, $e=\tilde{e}$, $c=\tilde{c}$,  $\tilde{f} \geq f$, and $\tilde{\mathfrak{g}}=\mathfrak{g}-1$, hence by the induction hypothesis $\tilde{v}-\tilde{e}+\tilde{f}-\tilde{c} \leq 1$, we see that $v-e+f-c \leq 1$.

If $\tilde{M}$ is disconnected, then we have two graphs embedded on surfaces: $(G_1, \tilde{M}_1)$ and $(G_2,\tilde{M}_2)$. If both $G_1$ and $G_2$ are non-empty, then $c_j+\mathfrak{g_j} \leq N$ for $j=1,2$, hence
\[
v_j -e_j+f_j - c_j \leq 1
\]
for $j=1,2$ by the induction hypothesis. Now because $v=v_1+v_2$, $e=e_1+e_2$, $f=f_1+f_2-1$, and $c=c_1+c_2$, we conclude that
\[
v-e+f-c \leq 1.
\]
Now, we consider the case that $G_2$ is empty. In this case, $\mathfrak{g}_2$ has to be greater than equal to $1$, because of the assumption that $C$ is not contractible. Hence $\mathfrak{g}_1 \leq \mathfrak{g} - 1$, and so we can apply the induction hypothesis to $(G_1, \tilde{M}_1)$. Because $v=v_1$, $e=e_1$, $f=f_1$, and $c=c_1$, this completes the proof.
\end{proof}

This lemma allows one to express the number of nodal domains in terms of the number of connected components of the nodal set, and the order of vanishing of singular points. 
\begin{theorem}
Let $M$ be a smooth compact Riemannian surface. Let $\phi$ be a real valued Laplacian eigenfunction on $M$. We denote by $C(\phi)$ the number of connected components of the nodal set $Z_\phi$, by $\Sigma(\phi)$ the set of all singular points of $\phi$. Then we have
\[
N(\phi) = C(\phi) +\sum_{p \in \Sigma(\phi)} (\mathrm{ord}_{x=p} \phi(x)-1) + O(1)
\]
\end{theorem}
\begin{proof}
As done in \cite{jz2016}, we may give a graph structure to the nodal set. In this context, we have $N(\phi) = f$, $C(\phi) = c$. To handle $v$ and $e$, observe that
\[
2e = \sum_{x \in V} deg x = 2v+\sum_{p \in \Sigma(\phi)} 2(\mathrm{ord}_{x=p} \phi(x)-1),
\]
hence
\[
e-v = \sum_{p \in \Sigma(\phi)} (\mathrm{ord}_{x=p} \phi(x)-1).
\]
Now the theorem follows by Lemma \ref{lemma1}.
\end{proof}
As a direct application, we give a short proof of an upper bound for the number of singular points of an eigenfunction on a surface, that is first obtained in \cite{df90} and \cite{dong92}:
\begin{corollary}
For a Laplacian eigenfunction $\phi$ on a surface with eigenvalue $\lambda$, we have
\[
\sum_{p \in \Sigma(\phi)} (\mathrm{ord}_{x=p} \phi(x)-1) =O(\lambda),
\]
as $\lambda \to \infty$.
\end{corollary}
\begin{proof}
Courant's general nodal domain theorem implies that $N(\phi_j) \leq j$, and Weyl's law implies that $j$ is asymptotically proportional to $\lambda_j$. Therefore we have
\[
\sum_{p \in \Sigma(\phi)} (\mathrm{ord}_{x=p} \phi(x)-1) \leq N(\phi) + O(1) = O(\lambda).
\]
\end{proof}
Another application is the following
\begin{corollary}
For an eigenfunction on a generic surface, we have
\[
N(\phi) = C(\phi) + O(1)
\]
as $\lambda \to \infty$.
\end{corollary}
\begin{proof}
It is proven in \cite{uh76} that eigenfunctions on a generic surface do not have any singular points.
\end{proof}
In particular, as done in previous sections, we may bound the number of nodal domains by counting the number of connected components of the nodal set on generic surfaces. However, the input that is missing in this case is the sharp upper bound for the number of intersection between the nodal set and a curve.

\bibliographystyle{alpha}
\bibliography{bibfile}

\end{document}